\documentclass[12pt,twoside]{amsart}
\usepackage{amsmath}
\usepackage{amsthm}
\usepackage{amsfonts}
\usepackage{amssymb}
\usepackage{latexsym}
\usepackage{mathrsfs}
\usepackage{amsmath}
\usepackage{amsthm}
\usepackage{amsfonts}
\usepackage{amssymb}
\usepackage{latexsym}
\usepackage{geometry}
\usepackage{dsfont}
\usepackage[dvips]{graphicx}
\usepackage[colorlinks=true,linkcolor=red,citecolor=blue]{hyperref}
\usepackage{color}
\usepackage[all]{xy}

\date{}
\allowdisplaybreaks[4] \footskip=15pt
\renewcommand{\uppercasenonmath}[1]{}

\numberwithin{equation}{section} \theoremstyle{plain}
\usepackage{graphicx,amssymb}
\usepackage[all]{xy}
\usepackage{amsmath}

\allowdisplaybreaks
\usepackage{amsthm}
\usepackage{color}

\theoremstyle{plain}
\newtheorem{theorem}{Theorem}[section]

\newtheorem{corollary}[theorem]{Corollary}
\newtheorem{example}[theorem]{Example}
\newtheorem*{open question}{Open Question}
\newtheorem{definition}[theorem]{Definition}

\theoremstyle{definition}

\theoremstyle{remark}
\newtheorem{remark}[theorem]{Remark}

\newcommand{\N}{\mathcal{N}}

\def\p{\frak p}

\def\Ker{{\rm Ker}}

\def\Im{{\rm Im}}

\def\Ann{{\rm Ann}}

\begin{document}
\begin{center}
{\large  \bf  A note on the Cohen type theorem and the Eakin-Nagata type theorem for uniformly $S$-Noetherian rings}

\vspace{0.5cm}
Xiaolei Zhang\\
\bigskip
School of Mathematics and Statistics, Shandong University of Technology,\\
Zibo 255049, China\\
E-mail: zxlrghj@163.com\\
\end{center}

\bigskip
\centerline { \bf  Abstract}
\bigskip
\leftskip10truemm \rightskip10truemm \noindent
In this note, we give the Cohen type theorem for uniformly $S$-Noetherian modules and the Eakin-Nagata type theorem for uniformly $S$-Noetherian rings. We also solve an open question proposed by   Kim and  Lim \cite[Question 4.10]{KL20}.
\\
\vbox to 0.3cm{}

\leftskip0truemm \rightskip0truemm
\bigskip
\section{Introduction}
Throughout this note, all rings are commutative rings with identity and all modules are unitary. Let $R$ be a ring. We always denote by $S$ a multiplicative subset of $R$, that is, $1\in S$ and $s_1s_2\in S$ for any $s_1\in S$, $s_2\in S$. let $M$ be an $R$-module. Denote by $\Ann_R(M)=\{r\in R\mid rM=0\}$. For a subset $U$ of  $M$,  denote by $\langle U\rangle$ the $R$-submodule of $M$ generated by $U$.

In the development of Noetherian rings,  Cohen type theorem and Eakin-Nagata type theorem are very crucial.
In the early 1950s, Cohen \cite{c50} showed that a ring $R$ is Noetherian if and only if every prime ideal of $R$ is finitely generated, which is called Cohen type theorem now. 
Recently, Parkash and  Kour \cite{pk21} generalized and extended  Cohen type theorem to Noetherian modules: a finitely generated $R$-module $M$ is Noetherian if and only if for every prime ideal $\p$ of $R$ with $\Ann(M)\subseteq \p$, there exists a finitely generated  submodule $N^\p$ of $M$ such that $\p M\subseteq N^\p\subseteq M(\p)$, where $M(\p):=\{x\in M \mid sx\in \p M $ for some $s\in R \setminus \p \}$. In the late 1960s, Eakin and
Nagata independently found that if $R\subseteq T$ be an extension of rings with $T$ a finitely generated $R$-module, then  $R$ is a Noetherian ring if and only if so is $T$ (see \cite{E68,N68}). And this well-known result is called Eakin-Nagata type theorem now.

In the past few decades, several generalizations of Noetherian rings (modules) have been extensively studied.  In 2002, Anderson and Dumitrescu \cite{ad02} introduced the notions of  $S$-Noetherian rings and $S$-Noetherian  modules. They also considered the Cohen type Theorem and  Eakin-Nagata type  Theorem for $S$-Noetherian rings
\cite[Proposition 4, Corollary 7]{ad02}.  Recently, Kim and Lim \cite{KL20} gave a  new proof of the Cohen type theorem for $S$-Noetherian modules and a  generalization of the Eakin-Nagata type theorem for $S$-Noetherian ring. They also  showed that if an $R$-module $M$ is faithful $S$-Noetherian with $S$ consisting of non-zero-divisors, then $R$ itself is an $S$-Noetherian ring, and latter they ask if the regularity of $S$ is essential? (see \cite[Proposition 3.7, Question 4.10]{KL20})

By noticing the elements chosen in $S$ in some concepts of $S$-versions of classical ones are not ``uniform'' in general, Zhang \cite{z21} recently introduced the notions of uniformly $S$-torsion modules, uniformly $S$-exact sequences etc. Utilizing the ``uniform'' ideas, Qi and Kim etc. \cite{QKWCZ} introduced the notions of uniformly $S$-Noetherian rings and uniformly $S$-Noetherian modules, and then distinguished them with the classical ones.  The main motivation of this paper is to investigate Cohen type Theorem and  Eakin-Nagata type Theorem for uniformly $S$-Noetherian rings and modules. More precisely, we showed that if  $S$ is anti-Archimedean, then an $R$-module $M$ is $u$-$S$-Noetherian if and only if there is an $s\in S$ such that $M$ is $s$-finite, and for every prime ideal $\p$ of $R$ with $\Ann_R(M)\subseteq \p$, there exists an $s$-finite submodule $N^\p$ of $M$ satisfying that $\p M\subseteq N^\p\subseteq M(\p)$ (see Theorem \ref{main});  and if $R\subseteq T$ be an extension of rings with $T$ an $S$-finite $R$-module, then  $R$ is an uniformly $S$-Noetherian ring if and only if so is $T$ (see Theorem \ref{en}). Moreover, we obtain that if there exists a faithful $R$-module $M$ which is also (resp., uniformly) $S$-Noetherian, then $R$ itself is an (resp., a uniformly) $S$-Noetherian ring, solving the open problem proposed by \cite[Question 4.10]{KL20} (See Theorem \ref{faith} and Theorem \ref{conj}).

\section{main results}
Let $R$ be a ring. Recall from \cite{ad02} that an $R$-module $M$ is \emph{$S$-finite} if for any submodule $N$ of $M$, there is an element $s\in S$ and a finitely generated $R$-module $F$ such that $sN\subseteq F\subseteq N$. In this case, we also say $M$ is $s$-finite.
Moreover, an $R$-module $M$ is called an \emph{$S$-Noetherian module} if every submodule of $M$ is $S$-finite, and a ring $R$ is called an \emph{$S$-Noetherian ring} if $R$ itself is an $S$-Noetherian $R$-module. Note that the choice of $s$ in these two concepts is decided by the submodules or ideals of the given module or ring.

To fill the gap of ``uniformity'' in the concept of $S$-Noetherian rings and $S$-Noetherian modules, the authors in \cite{QKWCZ} introduced the notions of uniformly $S$-Noetherian rings and uniformly $S$-Noetherian modules, and we restate them as follows.

\begin{definition}\label{us-no-module}\cite[Definition 2.1, Definition 2.6]{QKWCZ} Let $R$ be a ring and $S$ a multiplicative subset of $R$. An $R$-module $M$ is called a uniformly $S$-Noetherian $R$-module $($with respect to $s)$ provided the set of all submodules of $M$ is $s$-finite for some $s\in S$. A ring $R$ is called a uniformly $S$-Noetherian ring $($with respect to $s)$ if $R$ itself is a uniformly $S$-Noetherian $R$-module $($with respect to $s)$.
\end{definition}

We obviously have the following implications for both rings and modules:
$${\boxed{\mbox{Notherian}}}\Longrightarrow {\boxed{$u$\mbox{-}$S$\mbox{-Notherian}}}\Longrightarrow {\boxed{$S$\mbox{-Notherian}}}$$
However, the converses are not correct in general (see \cite[Example 2.2, Example 2.5]{QKWCZ} respectively). Recall that a multiplicative subset $S$ of $R$ is said to be anti-Archimedean if $\bigcap\limits_{n\geq 1}s^nR\bigcap S\not=\emptyset.$ The anti-Archimedean condition is very important in some results of $S$-Noetherian rings, such as Hilbert Theorem for $S$-Noetherian rings etc. (see
\cite[Proposition 9, Proposition 10]{ad02}). It is easy to verify that the  multiplicative set given in \cite[Example 2.5]{QKWCZ} is not anti-Archimedean. Now we give an example of $S$-Noetherian ring which is not uniformly $S$-Noetherian when $S$ is anti-Archimedean.

\begin{example}\label{exam-not-ut-1} Let $R$ be a valuation domain whose valuation group is the additive group $G=\mathbb{R}[x]$ of all polynomials with coefficients in the field $\mathbb{R}$ of real numbers, and the order is defined by $f(x)>0$ if its leading coefficient   $>0$. Let $S=R\setminus\{0\}$ the set of all nonzero elements of $R$. Then $S$ is anti-Archimedean, and $R$ is $S$-Noetherian but not uniformly  $S$-Noetherian.
 \end{example}
 \begin{proof}
First, we will show  $S$ is anti-Archimedean. Denote by $v$ the valuation of $R\setminus\{0\}$ to $G$. Let $s$ be a nonzero element in $R$.  Let $s'$ be an nonzero element in $R$ such that $\deg(v(s'))>\deg(v(s))$. Then we have $v(s')>nv(s)=v(s^n)$ for any positive integer $n$.  So $s'\in \bigcap\limits_{n\geq 1}s^nR\bigcap S$ for any $s\in S$, that is, $S$ is anti-Archimedean.

Then, we have that  $R$ is $S$-Noetherian.  Indeed, let $I$ be an nonzero ideal of $R$ and $0\not=s\in I$. Then $sI\subseteq sR\subseteq I$. It follows that  $R$ is $S$-Noetherian.

Finally, we claim that $R$ is not  uniformly  $S$-Noetherian. Suppose $R$ is  uniformly  $S$-Noetherian with respect to some $s\in S$. suppose $\deg(v(s))=n$. Then the $R_{s}$-ideal generated by $\{v^{-1}(x^{n+1}),v^{-1}(x^{n+2}),\dots\}$ is not finitely generated, where $R_s$ is the localization of $R$ at $S'=\{1,s,s^2,\dots\}$. So  $R_s$ is not Noetherian. Hence $R$ is not uniformly  $S$-Noetherian by \cite[Lemma 2.3]{QKWCZ}.
\end{proof}

Recently, Parkash and  Kour \cite{pk21} generalized and extended  Cohen type theorem to Noetherian modules: a finitely generated $R$-module $M$ is Noetherian if and only if for every prime ideal $\p$ of $R$ with $\Ann(M)\subseteq \p$, there exists a finitely generated  submodule $N^\p$ of $M$ such that $\p M\subseteq N^\p\subseteq M(\p)$, where $M(\p):=\{x\in M \mid sx\in \p M $ for some $s\in R \setminus \p \}$. Latter, Zhang \cite{ztwocohen} extended this result to $S$-Noetherian modules and $w$-Noetherian modules. In the following, we give the result for uniformly $S$-Noetherian  modules when $S$ is anti-Archimedean.

\begin{theorem}\label{main}\textbf{$($Cohen type theorem for uniformly $S$-Noetherian modules$)$} Let $R$ be a ring and $S$ an anti-Archimedean  multiplicative subset of $R$.  Then an $R$-module $M$ is uniformly $S$-Noetherian if and only if there exists $s\in S$ such that $M$ is $s$-finite, and for every prime ideal $\p$ of $R$ with $\Ann_R(M)\subseteq \p$, there exists an $s$-finite submodule $N^\p$ of $M$ satisfying that $\p M\subseteq N^\p\subseteq M(\p)$, where $M(\p)=\{x\in M \mid sx\in \p M $ for some $s\in R \setminus \p \}$.
\end{theorem}

\begin{proof} Suppose that $M$ is a uniformly $S$-Noetherian $R$-module. Then there is $s\in S$ such that the set of all submodules of $M$ is $s$-finite.
 Let $\p$ be a prime ideal with $\Ann_R(M)\subseteq \p$.  If we take $N^\p=\p M$, then $N^\p$ is certainly an $s$-finite submodule of $M$ satisfying $\p M\subseteq N^\p\subseteq M(\p)$.

On the other hand, let $s'\in \bigcap\limits_{n\geq 1}s^nR\bigcap S$. If $M$ is  uniformly $S$-Noetherian with respect to $s'$, then we are done. Otherwise, we will show  $M$ is  uniformly $S$-Noetherian with respect to $s^{n}$ for some positive integer $n$. On contrary, suppose  that $M$ is not uniformly $S$-Noetherian with respect to $s^k$ for any positive integer $k$. Let $\N$ be the set of all submodules of $M$ which are not $s^k$-finite for any positive integer $k$. We can assume $\N$ is non-empty.  Indeed, on contrary assume that for each submodule $N$ of $M$, there exists a nonnegative integer $k_N$  such that $N$ is $s^{k_N}$-finite. Since $S$ is anti-Archimedean, then there is an $s'\in \bigcap\limits_{n\geq 1}s^nR\bigcap S$ such that all submodules of $M$ are $s'$-finite. Hence  $M$ is uniformly $S$-Noetherian with respect to $s'$, and so the conclusion holds.

Make a partial order on $\N$ by defining  $N_1\leq N_2$ if and only if $N_1\subseteq N_2$ in $\N$. Let $\{N_i\mid i\in \Lambda\}$  be a chain in  $\N$. Set $N:=\bigcup\limits_{i\in \Lambda}N_i$. Then $N$ is not $s^k$-finite for any positive integer $k$. Indeed, suppose $s^{k_0}N\subseteq \langle x_1, \dots , x_n\rangle\subseteq N$ for some positive integer $k_0$. Then there exists $i_0\in \Lambda$ such that $\{ x_1, \dots , x_n\}\subseteq N_{i_0}$. Thus $s^{k_0}N_{i_0}\subseteq sN\subseteq \langle x_1, \dots , x_n\rangle\subseteq N_{i_0}$ implying that $N_{i_0}$ is $s^{k_0}$-finite, which is a contradiction.
By Zorn's Lemma  $\N$ has a maximal element, which is also denoted by $N$. Set $$\p:=(N:M)=\{r\in R \mid rM\subseteq N\}.$$

\textbf{(1) Claim that $\p$ is a prime ideal of $R$}. Assume on the contrary that there exist $a,b\in R \setminus \p$ such that $ab\in \p$. Since $a,b\in R \setminus \p$, we have $aM\not\subseteq N$ and $bM\not\subseteq N$. Therefore $N+aM$ is $s^{k_0}$-finite for some nonnegative integer $k_0$. Let $\{y_1,\dots,y_m\}$ be a subset of $N+aM$ such that $s^{k_0}(N+aM)\subseteq \langle y_1,\dots,y_m\rangle$. Write $y_i=w_i+az_i$ for some $w_i\in N$ and $z_i\in M\ (1\leq i\leq m)$. Set $L:=\{x\in M\mid ax\in N\}$. Then $N+bM\subseteq L$, and hence $L$ is also $s^{k_1}$-finite  for some nonnegative integer $k_1$. Let $\{x_1,\dots,x_k\}$ be a subset of $L$ such that $s^{k_1}L \subseteq\langle x_1,\dots,x_k \rangle$. Let $n \in N$ and write $$s^{k_0}n=\sum\limits_{i=1}^mr_iy_i=\sum\limits_{i=1}^mr_iw_i+a\sum\limits_{i=1}^mr_iz_i.$$
Then $\sum\limits_{i=1}^mr_iz_i\in L$. Thus $s^{k_1}\sum\limits_{i=1}^mr_iz_i=\sum\limits_{i=1}^kr'_ix_i$ for some $r'_i\in R$ ($i=1,\dots,k$).
So $$s^{k_0+k_1}n=\sum\limits_{i=1}^msr_iw_i+\sum\limits_{i=1}^kr'_iax_i.$$  And thus $s^{k_0+k_1}N\subseteq \langle w_1,\dots, w_m, ax_1,\dots, ax_k \rangle \subseteq N$ implying that $N$ is $s^{k_0+k_1}$-finite, which is a contradiction. Hence $\p$ is a prime ideal of $R$.

\textbf{(2) Claim that $M(\p)\subseteq N$}. Suppose on the contrary that there exists $y\in M(\p)$ such that $y\not\in N$. Then there exists $t\in R \setminus \p$ such that $ty\in \p M=(N:M)M\subseteq N$. As $t\not\in \p=(N:M)$,  it follows that $tM\not\subseteq N$. Therefore $N+tM$ is $s^{k_2}$-finite for some nonnegative integer $k_2$. Let $\{u_1,\dots,u_m\}$ be a subset of $N+tM$ such that $s^{k_2}(N+tM)\subseteq \langle u_1,\dots,u_m \rangle$ for some $s^{k_2}\in S$. Write $u_i=w_i+tz_i ~  (i=1,\dots,m)$ with $w_i\in N$ and $z_i\in M$. Set $T:=\{x\in M \mid tx\in N\}$. Then $N\subset N+Ry\subseteq T$, and hence $T$ is $s^{k_3}$-finite for some nonnegative integer $k_3$. Then there exists a subset $\{v_1,\dots,v_l\}$ of $T$ such that $s^{k_3}T\subseteq \langle v_1,\dots,v_l\rangle$. Let $n$ be an element in $N$. Then $$s^{k_2}n=\sum\limits_{i=1}^mr_iu_i=\sum\limits_{i=1}^mr_iw_i+t\sum\limits_{i=1}^mr_iz_i.$$ Thus $\sum\limits_{i=1}^mr_iz_i\in T$. So $s^{k_3}\sum\limits_{i=1}^mr_iz_i=\sum\limits_{i=1}^lr'_iv_i$ for some $r'_i\in R ~(i=1,\dots,l)$. Hence $s^{k_2+k_3}n=\sum\limits_{i=1}^ms_4r_iw_i+\sum\limits_{i=1}^lr'_itv_i.$ Thus $s^{k_2+k_3}N\subseteq \langle w_1,\dots, w_m, tv_1,\dots, tv_l \rangle $ implying that $N$ is $s^{k_2+k_3}$-finite, which is a contradiction. Hence $M(\p)\subseteq N$.

Finally, we will show $M$ is uniformly $S$-Noetherian. Since $M$ is $s$-finite, there exists a finitely generated submodule $F=\langle m_1, \dots , m_k\rangle$ of $M$ such that $sM\subseteq F$. Claim that $\p\cap S'=\emptyset$ where $S'=\{1,s,s^2,\cdots\}$. Indeed, if $s^{k_4}\in \p$ for  some nonnegative integer $k_4$, then $s^{k_4}M\subseteq N\subseteq M$. So $s^{1+k_4}N\subseteq s^{1+k_4}M\subseteq s^{k_4}F \subseteq s^{k_4}M \subseteq  N$ implies that $N$ is $s^{1+k_4}$-finite, which is a  contradiction. Note that  $$\p=(N:M)\subseteq (N:F)\subseteq(N:sM)=(\p:s)=\p$$ since $\p$ is a prime ideal of $R$. So $\p=(N:F)=(N:\langle m_1, \dots , m_k\rangle)=\bigcap\limits_{i=1}^k(N:Rm_i)$. By \cite[Proposition 1.11]{am69}, $\p=(N:Rm_j)$ for some $1\leq j\leq k$. Since $m_j\not\in N$, it follows that $N+Rm_j$ is $s^{k_5}$-finite for  some nonnegative integer $k_5$. Let $\{y_1,\dots,y_m\}$ be a subset of $N+Rm_j$ such that $s^{k_5}(N+Rm_j)\subseteq \langle y_1,\dots,y_m \rangle$. Write $y_i=w_i+a_im_j$ for some $w_i\in N$ and $a_i\in R ~(i=1,\dots,m)$. Let $n \in N$. Then $s^{k_5}n=\sum\limits_{i=1}^mr_i(w_i+a_im_j)=\sum\limits_{i=1}^mr_iw_i+(
\sum\limits_{i=1}^mr_ia_i)m_j$. Thus $(\sum\limits_{i=1}^mr_ia_i)m_j\in N $. So $\sum\limits_{i=1}^mr_ia_i\in \p$. Thus $s^{k_5}N\subseteq \langle w_1,\dots,  w_m \rangle +\p m_j$. As $\Ann_R(M)\subseteq (N:M)=\p$, there exists an $s$-finite submodule $N^\p$ of $M$ such that $\p M\subseteq N^\p\subseteq M(\p)$. Thus
\begin{eqnarray*}
  s^{k_5}N & \subseteq & \langle w_1,\dots, w_m \rangle + \p m_j \\
   & \subseteq & \langle w_1,\dots, w_m \rangle + \p M \\
   & \subseteq & \langle w_1,\dots,  w_m \rangle + N^\p \\
   &\subseteq & \langle w_1,\dots,  w_m \rangle +M(\p) \\
  & \subseteq & N
\end{eqnarray*}
Since $N^\p+\langle w_1,\dots,  w_m \rangle $ is $s$-finite, it follows that $N$ is  $s^{1+k_5}$-finite,  which is a  contradiction.
Consequently, we have $M$ is uniformly $S$-Noetherian with respect to $s^{k'}$ for some nonnegative integer $k'$.
\end{proof}

\begin{remark} We do not know whether the condition ``$S$ is anti-Archimedean'' in Theorem \ref{main} can be removed. Note that this condition is mainly use to show the set $\N$ in the proof of Theorem \ref{main} can be assumed to be non-empty.
\end{remark}

Taking $S=\{1\}$,  we can recover  Parkash and Kour's result.

\begin{corollary} \cite[Theorem 2.1]{pk21} Let $R$ be a ring and $M$ a finitely generated $R$-module. Then $M$ is Noetherian if and only if for every prime ideal $\p$ of $R$ with $\Ann_R(M)\subseteq \p$, there exists a finitely generated submodule $N^\p$ of $M$ such that $\p M\subseteq N^\p\subseteq M(\p)$.
\end{corollary}

There is a direct corollary of Theorem \ref{main}.
\begin{corollary}\label{cor-cohen} Let $R$ be a ring and $S$ an anti-Archimedean multiplicative subset of $R$.  Then an $R$-module $M$ is uniformly $S$-Noetherian if and only if there exists $s\in S$ such that $M$ is $s$-finite and $\p M$ is  $s$-finite for every prime ideal $\p$ of $R$.
\end{corollary}

The well-known  Eakin-Nagata type theorem  states that if $R\subseteq T$ be an extension of rings with $T$ a finitely generated $R$-module, then  $R$ is a Noetherian ring if and only if so is $T$ (see \cite{E68,N68}). Next, we give the Eakin-Nagata type theorem for uniformly $S$-Noetherian rings.

\begin{theorem}\label{en}\textbf{$($Eakin-Nagata type theorem for uniformly $S$-Noetherian rings$)$} Let $R$ be a ring, $S$ an anti-Archimedean multiplicative subset of $R$ and $T$ a ring extension of $R$. If $T$ is $S$-finite as an $R$-module. Then the following statements are equivalent.
\begin{enumerate}
\item $R$ is a uniformly $S$-Noetherian ring.
\item $T$ is a uniformly $S$-Noetherian ring.
\item There is $s\in S$ such that $\p T$ is an $s$-finite $T$-ideal for every prime ideal $\p$ of $R$.
\item $T$ is a uniformly $S$-Noetherian $R$-module.
\end{enumerate}
\end{theorem}
\begin{proof} $(1)\Rightarrow (2)$  Suppose $R$ is a uniformly $S$-Noetherian ring with respect to some $s_1\in S$. Let $I$ be an ideal of $T$. Since $R\subseteq T$, $I$ is an $R$-submodule of $T$. Suppose $T$ is $s_2$-finite as an $R$-module for some $s_2\in S$. Then $T$ is the image of a uniformly $S$-epimorphism $R^n\rightarrow T$. One can  use the proof of \cite[Lemma 2.12]{QKWCZ} to check $R^n$ is a uniformly $S$-Noetherian $R$-module with respect to $s_1^n$, . So $T$ is  a uniformly $S$-Noetherian $R$-module with respect to $s_1^ns_2$ by \cite[Proposition 2.13]{QKWCZ}. Then there exists $a_1,\dots,a_m\in I$ such that $s_1^ns_2I\subseteq \langle a_1,\dots,a_m\rangle R \subseteq I.$ Thus $s_1^ns_2I\subseteq \langle a_1,\dots,a_m\rangle T \subseteq I.$  Consequently, $T$
is a uniformly $S$-Noetherian ring with respect to $s_1^ns_2$.

$(2)\Rightarrow (3)$ Obvious.

$(3)\Rightarrow (4)$ Let $\p$ be a prime ideal  that satisfies $\Ann_R(T)\subseteq \p$. Then $\p T$ is $s$-finite as an $T$-ideal. So there exists $p_1,\dots,p_m\in \p$ such that $s(\p T)\subseteq \langle p_1,\dots,p_m\rangle T\subseteq \p T$. Since $T$ is $S$-finite, there exists $s'\in S$ and $t_1,\dots,t_n$ such that $s'T\subseteq \langle t_1,\dots,t_n\rangle R\subseteq T$. Therefore, we have
\begin{eqnarray*}
  s's(\p T) & \subseteq & s' \langle p_1,\dots,p_m\rangle T  \\
   & = &  s'p_1T+\cdots+s'p_mT  \\
   & \subseteq& p_1(t_1R+\cdots t_nR)+\cdots+ p_m(t_1R+\cdots t_nR)\\
   &\subseteq & \p T
\end{eqnarray*}
Hence $\p T$ is $s's$-finite as an $R$-module. It follows by  Corollary \ref{cor-cohen} that $T$ is a uniformly $S$-Noetherian $R$-module.

$(4)\Rightarrow (1)$ Suppose $T$ is a uniformly $S$-Noetherian $R$-module. Since $R$ is an $R$-submodule of $T$, $R$ is also a   uniformly $S$-Noetherian $R$-module by \cite[Lemma 2.12]{QKWCZ}. It follows that $R$ is a uniformly $S$-Noetherian ring.
\end{proof}

Let $R$ be a ring and $M$ an $R$-module. Recall that $M$ is faithful if $\Ann_R(M)=0.$ We say $M$ is  $S$-faithful if $t\Ann_R(M)=0$ for some $t\in S$. Hence faithful $R$-modules are all $S$-faithful. It is well-known that if a faithful $R$-module $M$ is Noetherian, then $R$ itself is a Noetherian ring (see \cite[Exercise 2.32]{fk16}).

\begin{theorem}\label{faith} Let $R$ be a ring, $S$ a multiplicative subset of $R$ and $M$ an $S$-faithful $R$-module. If $M$ is a uniformly $S$-Noetherian $R$-module, then $R$ is a uniformly $S$-Noetherian ring.
\end{theorem}
\begin{proof} Suppose $M$ is a uniformly $S$-Noetherian $R$-module with respect to some $s\in S$. Then $M$ is $s$-finite, and so there exists $m_1,\dots,m_n\in M$ such that $sM\subseteq \langle m_1,\dots,m_n\rangle\subseteq M$. Consider the $R$-homomorphism $\phi:R\rightarrow M^n$ given by $\phi(r)=(rm_1,\dots,rm_n)$. We claim that $s\Ker(\phi)=0$. Indeed, let $r\in \Ker(\phi)$. Then $rm_i=0$ for each $i=1,\dots,n$. Hence $srM\subseteq r\langle m_1,\dots,m_n\rangle=0$. And hence $sr\in \Ann_R(M)$. Since $M$ is an $S$-faithful $R$-module, we have $tsr=0$ for some $t\in S$, and so $ts\Ker(\phi)=0$. Note that the $R$-module $M^n$ is uniformly $S$-Noetherian with respect to  $s^n$ by the proof of \cite[Lemma 2.12]{QKWCZ}. Hence the $R$-module $\Im(\phi)$ is also uniformly $S$-Noetherian with respect to  $s^n$. Considering the exact sequence $$0\rightarrow \Ker(\phi)\rightarrow R \rightarrow \Im(\phi) \rightarrow0,$$  we have $R$ is a uniformly $S$-Noetherian ring with respect to  $ts^{n+1}$.
\end{proof}

Recently, the authors in \cite[Proposition 3.7]{KL20} showed that Theorem \ref{faith} also holds for $S$-Noetherian ring (modules) when $S$ consists of non-zero-divisors, and ask if the condition ``$S$ consists of non-zero-divisors'' is essential (see \cite[Question 4.10]{KL20}). Inspired by the proof of Theorem \ref{faith}, we can show the condition ``$S$ consists of non-zero-divisors'' in \cite[Proposition 3.7]{KL20} can be removed .

\begin{theorem}\label{conj} Let $R$ be a ring, $S$ a multiplicative subset of $R$ and $M$ an $S$-faithful  $R$-module $($for example,  $M$ is a faithful $R$-module$)$. If $M$ is an $S$-Noetherian $R$-module, then $R$ is an $S$-Noetherian ring.
\end{theorem}
\begin{proof} Let $M$ be an $S$-Noetherian faithful  $R$-module. Then $M$ is $S$-finite, and so there exist $s\in S$ and  $m_1,\dots,m_n\in M$ such that $sM\subseteq \langle m_1,\dots,m_n\rangle\subseteq M$. Consider the $R$-homomorphism $\phi:R\rightarrow M^n$ given by $\phi(r)=(rm_1,\dots,rm_n)$. We claim that $s\Ker(\phi)=0$. Indeed, let $r\in \Ker(\phi)$. Then $rm_i=0$ for each $i=1,\dots,n$. Hence $srM\subseteq r\langle m_1,\dots,m_n\rangle=0$. And hence $sr\in \Ann_R(M)$. Since $M$ is an $S$-faithful $R$-module, we have $tsr=0$ for some $t\in S$, and so $ts\Ker(\phi)=0$.  Note that $M^n$ is also an $S$-Noetherian $R$-module, and so is its submodule $\Im(\phi)$.  Let $I$ be an ideal of $R$. Then $\phi (I)$ is a submodule of $\Im(\phi)$, and so is $S$-finite. Thus there exists $s'\in S$ and $r_1,\cdots r_n\in I$ such that
$$s'\phi(I)\subseteq \phi(r_1R+\cdots+r_nR) \subseteq \phi(I).$$
We claim that $ss'I\subseteq r_1R+\cdots+r_nR$. Indeed, for any $x\in I$, we have $s'\phi(x)=\phi(r_1t_1+\cdots+r_nt_n)$ for some $t_i\in R\ (i=1,\dots,n).$ Hence $\phi(r_1t_1+\cdots+r_nt_n-s'x)=0$. So
$r_1t_1+\cdots+r_nt_n-s'x\in \Ker(\phi)$, and thus $ts(r_1t_1+\cdots+r_nt_n)-tss'x=0$. It follows that  $tss'I\subseteq ts (r_1R+\cdots+r_nR)\subseteq r_1R+\cdots+r_nR\subseteq I$. Hence $I$ is $S$-finite. So $R$ is an $S$-Noetherian ring.
\end{proof}

\end{document}